\documentclass[a4paper,11pt]{article}
\usepackage{amsmath,amsfonts,amssymb}
\usepackage[thmmarks,amsmath]{ntheorem}
\usepackage{enumerate}
\usepackage{a4wide}
\usepackage{graphics}
\usepackage{epstopdf}
\usepackage{epsfig}
\usepackage{curves}
\usepackage{epic,eepic}
\usepackage{ntheorem}
\usepackage{lscape}
\usepackage{float}
\floatstyle{boxed}
\usepackage{authblk}
\allowdisplaybreaks

\newtheorem{proposition}{Proposition}[section]
\newtheorem{theorem}[proposition]{Theorem}
\newtheorem{corollary}[proposition]{Corollary}

\floatstyle{ruled}
\newfloat{algorithm}{tbp}{loa}
\providecommand{\algorithmname}{Algorithm}
\floatname{algorithm}{\protect\algorithmname}

\newcounter{algo}

\floatstyle{ruled}
\newfloat{algorithmbis}{tbp}{loa}
\providecommand{\algorithmname}{Algorithm}
\floatname{algorithmbis}{\protect\algorithmname}

\newcounter{algobis}

\theorembodyfont{\normalfont}
\newtheorem{example}[proposition]{Example}
\newtheorem{remark}[proposition]{Remark}

\newenvironment{proof}{\noindent\textbf{Proof.}
  }{\hspace*{\fill}$\Box$ \\[1em]}

\renewcommand{\Re}{{\mathbb R}}

\def\trt{^{\scriptscriptstyle T}}
\title{A Bridge between Bilevel Programs and Nash Games}

\author[1]{Lorenzo Lampariello} 
\author[2]{Simone Sagratella}

\affil[1]{Department of Business Studies, Roma Tre University, Via D'Amico 77, 00145 Roma, Italy\\

lorenzo.lampariello@uniroma1.it.}

\affil[2]{Department of Computer, Control and Management Engineering Antonio Ruberti, Sapienza University of Rome, Via Ariosto 25, 00185 Roma, Italy\\

sagratella@dis.uniroma1.it. The work of this author's research has been partially supported by Avvio alla Ricerca 2015 Sapienza University of Rome, under grant 488.}

\begin{document}

\maketitle

\begin{abstract}
We study connections between optimistic bilevel programming problems and Generalized Nash Equilibrium Problems (GNEP)s.
%We remark that, when addressing bilevel problems, we consider the general case in which the lower level program is not assumed to have a unique solution.
Inspired by the optimal value approach, we propose a new GNEP model that incorporates some taste of hierarchy and turns out to be related to the bilevel program.
We provide a complete theoretical analysis of the relationship between the vertical bilevel problem and our ``uneven'' horizontal model: we define classes of problems for which solutions of the bilevel program can be computed by finding equilibria of our GNEP.
Furthermore, from a modelistic standpoint, by referring to some applications in economics, we show that our ``uneven'' horizontal model lies between the vertical bilevel model and a ``pure'' horizontal  game.
%We develop a simple algorithm, which turns out to be globally convergent, for the solution of classes of our GNEP; we study how and when it is then possible to recover a solution of the bilevel problem from the computed equilibrium.%, even when addressing big nonlinear bilevel problems. 

\end{abstract}

\section{Introduction}
We aim at building a bridge between optimistic bilevel programming problems and generalized Nash equilibrium problems. This kind of study, as far as we are aware, has never been considered in the literature. In particular, we wish to point out differences and similarities between two-level optimization and one-level game models. Besides being of independent theoretical and modelistic interest, this analysis gives a new perspective on bilevel problems.     

Bilevel programming is a fruitful modeling framework that is widely used in many fields, ranging from economy and engineering to natural sciences (see \cite{colson2007overview}, the fundamental \cite{dempe2002foundations}, \cite{dempe2003annotated}, the recent \cite{dempe2013bilevel}, the references therein, the seminal paper \cite{von1934marktform}, and \cite{aussel2013electricity,hu2013existence} for recent applications). This problem has a hierarchical structure involving two decision, {\textit{upper}} and {\textit{lower}}, levels. We focus on the more general and challenging case in which the lower level program is not assumed to have a unique solution. We recall that, whenever lower level solutions are non-uniquely determined, the definition itself of the bilevel program is ambiguous. With this in mind, in this work we refer to the most common optimistic vision. Roughly speaking, in optimistic bilevel problems a decision is taken, at the {\textit{upper}} level, by considering two blocks of variables, namely $x$ and $y$; but, in turn, $y$ is implicitly constrained by the reaction of a subaltern ({\textit{lower}} level) part to the choice of $x$. Thus, bilevel programs can be viewed, in some sense, as a special two-agents optimization. The two agents play here an asymmetric role, in that the variable block $x$ is controlled only by the upper level agent, while the choice of the second block $y$ is influenced by both the upper and the lower level agents. It is precisely this asymmetrically shared influence on the variable blocks that makes bilevel problems inherently hard to solve. It is worth noting that, whenever there is not such a thorny relationship between the agents, things become {\em conceptually} simpler. Indeed, on the one hand, if all the variables are controlled by both the agents, we have a pure hierarchical problem (in Section \ref{sec:ref} we show that this problem has the same set of solutions of a suitable one-level generalized Nash equilibrium problem); while, on the other hand, with $x$ being controlled by the upper level agent, if $y$ is controlled only by the lower level agent, we get a generalized Nash equilibrium problem, in which the two agents act as players at the same level (see Section \ref{sec:backg}).

Optimistic bilevel problems have been studied in two different versions (see \cite{zemkoho2014solving} for a rather complete discussion on this topic): the Original optimistic Bilevel programming Problem (OBP)
\begin{equation}\label{ex: origoptim}
\begin{array}{cl}
\underset{x}{\mbox{minimize}} & \min_y\{F(x, y): \, y \in S(x)\} \\
\mbox{s.t.} & x \in X,
\end{array}
\end{equation}
and the Standard optimistic Bilevel programming Problem (SBP)
\begin{equation}\label{eq: bilevel}
\begin{array}{cl}
\underset{x,y}{\mbox{minimize}} & F(x,y)\\
\mbox{s.t.} & x \in X\\[5pt]
 & y \in S(x),
\end{array}
\end{equation}
where $F: \Re^{n_1} \times \Re^{n_2} \to \Re$, $X \subseteq \Re^{n_1}$ and the set-valued mapping $S: \Re^{n_1} \rightrightarrows \Re^{n_2}$ describes the solution set of the following lower level parametric optimization problem:
\begin{equation}\label{eq: follower}
\begin{array}{cl}
\underset{w}{\mbox{minimize}} & f(x,w)\\
\mbox{s.t.} & w \in U\\[5pt]
& g(x,w) \leq 0,
\end{array}
\end{equation}
where $f: \Re^{n_1} \times \Re^{n_2} \to \Re$ and $g: \Re^{n_1} \times \Re^{n_2} \to \Re^m$ and $U \subseteq \Re^{n_2}$.

As observed in \cite{dempe2012sensitivity,zemkoho2014solving}, OBP and SBP are equivalent in the global case but a local minimum of SBP may not lead to a local solution of OBP. We underline that, besides \cite{zemkoho2014solving}, which deals with OBPs, almost all other solution methods cope only with SBPs. The latter problems are structurally nonconvex and nonsmooth (see \cite{dempe2007new}); furthermore, it is hard to define suitable constraint qualification conditions for them, see, e.g., \cite{dempe2011generalized,ye2006constraint}. In fact, the study of provably convergent and practically implementable algorithms for the solution of even just SBPs is still in its infancy (see, for example, \cite{bard1983algorithm,dempe2003annotated,dempe2014solution,dempe2015solution,lin2014solving,mitsos2008global,outrata1990numerical,solodov2007explicit,vicente1994bilevel,xu2014smoothing,zemkoho2014solving}), as also witnessed by the scarcity of results in the literature. We remark that suitable reformulations of the SBP have been proposed in order to investigate optimality conditions and constraint qualifications, as well as to devise suitable algorithmic approaches: to date, the most studied and promising are optimal value and KKT one level reformulations (see \cite{dempe2013bilevel}, the references therein and \cite{outrata1988note,ye2010new}). As far as the KKT reformulation is concerned, it should be remarked that the SBP has often be considered as a special case of Mathematical Program with Complementarity Constraints (MPCC) (see, e.g., \cite{facchinei1999smoothing,fletcher2006local,luo1996mathematical}). Actually, this is not the case, as shown in \cite{dempe2012bilevel}. Indeed, in general, one can provably recast the SBP as an MPCC only when the lower level problem is convex and Slater's constraint qualification holds for all $x$. Moreover, even in this case, a local solution of the MPCC, which is what one can expect to compute (since the MPCC is nonconvex), may happen not to be a local optimal solution of the corresponding SBP and, even less, OBP.

Generalized Nash Equilibrium Problem (GNEP) is another important modeling tool in multi-agent contexts. GNEPs, that, unlike SBPs, are problems in which all agents act at the same level, have been extensively studied in the literature and many methods have been proposed for their solutions in the last decades, see, e.g., \cite{dreves2011solution,facchinei2011partial,facchlampsagr2011sur,facchinei2011computation,FukPang09,sagratella2015computing}. For further details, we refer the interested reader to \cite{FacchKanSu}. Finally, we would like to cite the interesting paper \cite{dorsch2013intrinsic}, which deals with both bilevel problems and GNEPs but without establishing connections between them, as we do.

In this work, building on the ideas set forth in \cite{lampariello2015matter}, we propose a new suitable GNEP model that is closely related to the SBP and proves to be connected with the OBP also. Our GNEP model is, in some sense, inspired by the optimal value approach, in that, when passing from the vertical structure of bilevel problems to the horizontal format of GNEPs, we exploit the value function idea to mimic the original relationship between the agents. Thus, despite its one-level structure, the latter GNEP incorporates some taste of hierarchy. 

To be more specific, here we summarize the theoretical results about the relationship between SBP/OBP and our GNEP model. In Theorem \ref{th: solution sets inclusion} we show that an equilibrium of our GNEP gives a feasible and, at least, suboptimal (possibly global optimal under some suitable conditions) solution for the corresponding SBP. With Proposition \ref{pr: easy points}, we define a particular type of global solutions of the SBP that, in any case, can be computed by finding an equilibrium of our GNEP. With Corollary \ref{co: g vanish} and with Theorem \ref{th: solution set equiv}, we identify classes of problems (including Stackelberg games and pure hierarchical optimization problems, see Remarks \ref{rm:Stack} and \ref{rm: nollx}, respectively) for which an equilibrium of our GNEP always leads to a global solution of the SBP. We remind that global solutions of the SBP lead also to global solutions of the OBP. Thus, the previous relations hold also between equilibria and global optima of the OBP. In Subsection \ref{sec:locstr}, we introduce the concept of strong local minima of the SBP: unlike general local solutions of the SBP, strong local minima enjoy the nice property to lead also to local solutions of the OBP (see Proposition \ref{th: strlocoptim}). With Theorem \ref{th: local solution sets inclusion} we give sufficient conditions for an equilibrium of our GNEP to lead to a strong local minimum of the SBP and, thus, also to a local minimum of the OBP. Section \ref{sec:ref} is equipped with several examples: in particular, we wish to cite Example \ref{ex:dempedutta} in which we compare our GNEP to the classical MPCC reformulation.

Relying on the previous theoretical results, in Section \ref{sec: application} we consider some applications in economics to show that our ``uneven'' horizontal framework, in some sense, lies between the vertical bilevel model and a ``pure'' horizontal  game. In a market with two firms producing some goods, we study the system's behavior in terms of outcomes values by employing three different points of view: vertical (for which a firm is the leader and the other one is the follower), horizontal (for which both firms act at the same level) and our uneven horizontal.

%Based on the previous considerations, Subsection \ref{subsec:alg1 general} deals with the practical solution of our GNEP: there we define classes of problems for which the proposed method is provably convergent to an equilibrium. Then, in Subsection \ref{subsec:alg2} we investigate whether and how it is possible to obtain a solution of the SBP/OBP from an equilibrium of the GNEP.  

%To the best of our knowledge, only very few methods can deal with problems in this general setting \cite{dempe2015solution,lin2014solving,outrata1990numerical,xu2014smoothing,zemkoho2014solving}, but none of them are easily applicable to big problems.

%The paper is organized as follows. In Section \ref{sec:backg} we recall some useful basic notions and we introduce a qualitative preliminary comparison between SBPs and GNEPs.  Section \ref{sec:ref} is devoted to our new GNEP model and its relationship with the SBP/OBP, while in Section \ref{sec:frame} we present the algorithmic framework and in Section \ref{sec:num} we report some preliminary numerical experiments.

\section{Preliminaries}\label{sec:backg}

We briefly recall some basic facts.
When dealing with SBP/OBP \eqref{eq: bilevel}/\eqref{ex: origoptim} we rely on the following standard assumptions: $F, f: \Re^{n_1} \times \Re^{n_2} \to \Re$ and $g: \Re^{n_1} \times \Re^{n_2} \to \Re^m$ are continuous, and $X \subseteq \Re^{n_1}$ and $U \subseteq \Re^{n_2}$ are closed.

Let $W \triangleq \{(x,y): \, x \in X, \, y \in S(x)\}$ and $U \, \cap \, K(x)$, with $K(x) \triangleq \left\{v \in \Re^{n_2}: \, g(x,v) \leq 0\right\}$, denote the feasible sets of SBP \eqref{eq: bilevel} and of lower level problem \eqref{eq: follower}, respectively. 

A point $(x^*, y^*)$ is a global solution of SBP \eqref{eq: bilevel} if $(x^*,y^*) \in W$ and $F(x^*,y^*) \leq F(x,y),$ $\forall \, (x,y) \in W$. More explicitly, feasibility and optimality of $(x^*, y^*)$ can be equivalently rewritten in the following manner:  
\begin{align}
&(x^*,y^*) \in X \times U, \; f(x^*,y^*) \leq f(x^*,y) \; \forall y \in U \cap K(x^*), \; g(x^*,y^*) \leq 0 \label{eq: bilevel solution feasib expl}\\[5pt]
&F(x^*,y^*) \leq F(x,y) \; \forall \, (x,y) \in W,
 \label{eq: bilevel solution optim expl}
\end{align}
where $W = \big\{ (u,v) \in X \times U:\; f(u,v) \leq f(u,w) \; \forall w \in U \cap K(u), \; g(u,v) \leq 0 \big\}$.

We would like to mention two particularly interesting and well-studied classes of SBPs: (optimistic) Stackelberg games and pure hierarchical optimization problems. Stackelberg games are SBPs in which function $g$ does not depend on the upper variables $x$. On the other hand, when, at the lower level, the whole dependence on $x$ is dropped, the SBP boils down to the following pure hierarchical optimization problem:
\begin{equation}\label{eq: hier}
\begin{array}{cl}
\underset{x,y}{\mbox{minimize}} & F(x,y) \\
\mbox{s.t.} & x \in X \\
& y \in S,
\end{array}
\end{equation}
where $S$ denotes the solution set of the lower level problem
$$
\begin{array}{cl}
\underset{w}{\mbox{minimize}} & f(w)\\
& w \in U\\
& g(w) \le 0.
\end{array}
$$
As we have pointed out in the introduction, the characteristic aspect of SBP \eqref{eq: bilevel} is the hierarchical relationship between the leader and the follower: the two agents play here an asymmetric role, in that the variable block $x$ is controlled only by the upper level agent, while the second block $y$ is controlled by both the upper and the lower level agents. Question arises naturally on what happens if the leader loses control on $y$. In the latter case, we get the following GNEP:
\begin{equation}\label{eq: gnepnaif}
\begin{array}{clccl}
\underset{x}{\mbox{minimize}} & \; F(x,y) & \hspace{70pt} & \underset{y}{\mbox{minimize}} & \; f(x,y)\\
\mbox{s.t.} & x \in X &  & \mbox{s.t.} & y \in U\\[5pt]
 & & & & g(x,y)\leq 0.
\end{array}
\end{equation}
Note that in GNEP \eqref{eq: gnepnaif} the two agents are at the same level, unlike in SBP \eqref{eq: bilevel}. 

One may think that in problem \eqref{eq: gnepnaif} the follower has been promoted at an upper level, the same of the leader; but this is not the case: indeed, the follower acts in the same manner in \eqref{eq: bilevel} and in \eqref{eq: gnepnaif}. Is the leader who is downgraded at the follower's level: in fact, unlike problem \eqref{eq: gnepnaif}, where the leader can no longer directly control $y$, in SBP \eqref{eq: bilevel} the follower is like ``a puppet in leader's hands''. 

Finally, we denote by ${\mathcal{N}}(\bar x)$ the collection of open neighborhoods of $\bar x$ and by $\text{dom} \, M \triangleq \{x \, | \, M(x) \neq \emptyset\}$ the domain of $M: \mathbb R^n  \rightrightarrows \mathbb R^m$.

\section{Taking care of hierarchy: a new GNEP model}\label{sec:ref}

In the light of the observations in Section \ref{sec:backg}, we propose to address a GNEP that better takes into account the original hierarchy between agents. With the following GNEP, we aim at positioning the leader in an intermediate level between that in \eqref{eq: gnepnaif} and that in \eqref{eq: bilevel}.         
\begin{equation}\label{eq: gnep}
\begin{array}{clccl}
\underset{x,y}{\mbox{minimize}} & \; F(x,y) & \hspace{70pt} & \underset{w}{\mbox{minimize}} & \; f(x,w)\\
\mbox{s.t.} & (x,y) \in X \times U &  & \mbox{s.t.} & w \in U\\[5pt]
 & f(x,y) \leq f(x,w) & & & g(x,w)\leq 0.\\[5pt]
 & g(x,y) \leq 0 & & &
\end{array}
\end{equation}
We say that the player controlling $x$ and $y$ is the leader, while the other player is the follower. Note that, in the leader's problem, only the feasible set, in particular constraint $f(x,y) \le f(x,w)$, depends on the follower's variables $w$; on the other hand, as regards follower's problem, the coupling with the leader's strategy may happen at both the objective and the feasible set levels. 

GNEP \eqref{eq: gnep} is related to the SBP/OBP, as the forthcoming considerations clearly show (see Theorems \ref{th: solution sets inclusion}, \ref{th: solution set equiv}, \ref{th: local solution sets inclusion}, Proposition \ref{pr: easy points}, Corollary \ref{co: g vanish} and Examples \ref{ex:GNBP} and \ref{ex:BPGN}). We point out that, in order to devise GNEP \eqref{eq: gnep}, we draw inspiration from the optimal value approach (see \cite{dempe2013bilevel,outrata1988note,ye2010new}). Indeed, the structure of leader's feasible set in \eqref{eq: gnep} (in particular, constraint $f(x, y) \le f(x, w)$) is intended to mimic, in some sense, and to deal with the value function implicit constraint $f(x, y) \le \varphi(x)$, where
\begin{equation*}\label{eq:valfunc}
\varphi(x) \triangleq \min_y\{f(x, y): \, y \in K(x) \cap U\}
\end{equation*}
is the value function. In problem \eqref{eq: gnep} the leader takes back control of variables $y$: this fact and the presence of constraint $f(x,y) \le f(x,w)$, introducing some degree of hierarchy in a level playing field, keep memory of the original balance of power between leader and follower.

We note that, as one can expect, it is precisely the ``difficult'' constraint $f(x, y) \le f(x, w)$ that makes, in general, problem \eqref{eq: gnep} not easily solvable: because of the presence of such constraint, GNEP \eqref{eq: gnep} may lack convexity and suitable constraint qualifications are not readily at hand. However, as will become evident in the subsequent sections, one can still define classes of bilevel problems for which problem \eqref{eq: gnep} is practically solvable. 

Moreover, in view of the above considerations, GNEP \eqref{eq: gnep} may also be considered as an alternative modeling tool, of independent interest, for describing systems in which there is a hierarchical interaction between agents. 

We denote by
$$
T \triangleq \left\{(x,y) \in X \times U \, : \, g(x,y) \le 0 \right\} \;\; \text{and} \;\; U
$$
the ``private'' constraints sets, and by
$$
H(w) \triangleq \left\{(x,y) \in \Re^{n_1} \times \Re^{n_2} \, : \, f(x,y) \le f(x,w) \right\} \;\; \text{and} \;\; K(x)
$$
the ``coupling'' constraints sets of the leader and the follower, respectively. Moreover, let $V(w) \triangleq T \cap H(w)$ be the feasible set of the leader. 

A solution, or an equilibrium, of GNEP \eqref{eq: gnep} is a triple $(x^*, y^*, w^*)$ such that
\begin{align}
 &(x^*,y^*) \in X \times U, \;\; f(x^*,y^*) \leq f(x^*,w^*), \;\; g(x^*,y^*) \leq 0, \label{eq: gnep solution feasib leader}\\[5pt]
 &F(x^*,y^*) \leq F(x,y), \; \forall \, (x,y) \in V(w^*), \label{eq: gnep solution optim leader}\\[5pt]
 & w^* \in U, \;\; g(x^*,w^*) \leq 0, \label{eq: gnep solution feasib follower}\\[5pt]
 &f(x^*,w^*) \leq f(x^*,w), \;\; \forall w \in U \cap K(x^*), \label{eq: gnep solution optim follower}
\end{align}
where $V(w^*) = \big\{ (u,v) \in X \times U: \; f(u,v) \leq f(u,w^*), \; g(u,v) \leq 0 \big\}$.
Conditions \eqref{eq: gnep solution feasib leader}-\eqref{eq: gnep solution optim leader} and \eqref{eq: gnep solution feasib follower}-\eqref{eq: gnep solution optim follower} state feasibility and optimality of $(x^*, y^*, w^*)$ for leader's problem and for follower's problem, respectively.

\subsection{Global solutions}
The following Theorem \ref{th: solution sets inclusion} allows us to establish relations between equilibria of GNEP \eqref{eq: gnep} and global solutions of  SBP \eqref{eq: bilevel} and, thus, of OBP \eqref{ex: origoptim}. On the one hand, Theorem \ref{th: solution sets inclusion} gives a sufficient condition for an equilibrium of GNEP \eqref{eq: gnep} to lead to a global solution of the SBP/OBP; on the other hand, as the subsequent developments in this section clearly show, it provides a theoretical base to define classes of bilevel problems that are tightly connected to the GNEP (see Corollary \ref{co: g vanish}, Theorem \ref{th: solution set equiv}, and Remarks \ref{rm:Stack} and \ref{rm: nollx}).
\begin{theorem}\label{th: solution sets inclusion}
Let $(x^*, y^*, w^*)$ be an equilibrium of GNEP \eqref{eq: gnep}. Then
\begin{enumerate}%[(i)]
 \item[(i)] $(x^*, y^*)$ is a feasible point for SBP \eqref{eq: bilevel}, that is $(x^*, y^*) \in W$;
 \item[(ii)] if $g(x,w^*)\leq 0$ for all $x$ such that there exists $y$ with $(x,y) \in W$ and $F(x,y) \leq F(x^*,y^*)$, then $(x^*, y^*)$ is a global solution of SBP \eqref{eq: bilevel}.
\end{enumerate}
\end{theorem}

\begin{proof}
Under the assumptions of the theorem, $(x^*, y^*, w^*)$ satisfy relations \eqref{eq: gnep solution feasib leader}-\eqref{eq: gnep solution optim follower}.

(i) We observe that \eqref{eq: gnep solution feasib leader}, \eqref{eq: gnep solution feasib follower} and \eqref{eq: gnep solution optim follower} together imply that $(x^*, y^*)$ satisfies \eqref{eq: bilevel solution feasib expl}, that is $(x^*, y^*) \in W$.

(ii) We need to show that \eqref{eq: bilevel solution optim expl} holds at $(x^*, y^*)$. Let us denote by ${\cal L}^*$ the level set of $F$ at $(x^*,y^*)$, and by $({\cal L}^*)^c$ its complement:
\begin{equation}\label{eq:L}
{\cal L}^*\triangleq \left\{ (x,y) \in \Re^{n_1} \times \Re^{n_2}: F(x,y) \leq F(x^*,y^*) \right\},
\end{equation}
\begin{equation}\label{eq:Lc}
({\cal L}^*)^c \triangleq \left\{ (x,y) \in \Re^{n_1} \times \Re^{n_2}: F(x,y) > F(x^*,y^*) \right\}.
\end{equation}
Let $(\bar x, \bar y)$ be any couple in $W \cap {\cal L}^*$: by assumptions, we have $g(\bar x,w^*)\leq 0$. Therefore, $w^* \in U \cap K(\bar x)$ and, since $(\bar x, \bar y) \in W$, in turn $(\bar x, \bar y) \in V(w^*)$ and
\begin{equation}\label{eq: in proof th 1}
W \cap {\cal L}^* \subseteq V(w^*).
\end{equation}
Thanks to \eqref{eq: gnep solution optim leader} and \eqref{eq: in proof th 1}, and noting that for every $(x,y) \in W \cap ({\cal L}^*)^c$ we have $F(x,y) > F(x^*,y^*)$, \eqref{eq: bilevel solution optim expl} holds at $(x^*, y^*)$. Hence, $(x^*, y^*)$ is a global solution of SBP \eqref{eq: bilevel}. 
\qed \end{proof}
It is worth noticing that condition (ii) also suggests that $(x^*,y^*)$ can be interpreted as a suboptimal point for SBP \eqref{eq: bilevel}. Indeed, we have $F(x^*,y^*) \le F(x,y)$ for every $(x,y) \in W$ with $x$ such that $g(x, w^*) \le 0.$

The following example gives a picture of the relationship between GNEP \eqref{eq: gnep} and SBP \eqref{eq: bilevel}, as stated in Theorem \ref{th: solution sets inclusion}.
\begin{example}\label{ex:GNBP}
Let us consider the following SBP:
\begin{equation}\label{ex: bilev}
\begin{array}{cl}
\underset{x,y}{\mbox{minimize}} & x^2+y^2 \\
\mbox{s.t.} & x \geq 1\\[5pt]
 & y \in S(x),
\end{array}
\end{equation}
where $S(x)$ denotes the solution set of the lower level problem
$$
\begin{array}{cl}
\underset{w}{\mbox{minimize}} & w\\
& x+w \geq 1,
\end{array}
$$
and the corresponding GNEP, that is,
\begin{equation}\label{ex: gnep}
\begin{array}{clccl}
\underset{x,y}{\mbox{minimize}} & \; x^2+y^2 & \hspace{70pt} & \underset{w }{\mbox{minimize}} & \; w\\
\mbox{s.t.} & x \geq 1 &  & \mbox{s.t.} & x+w \geq 1. \\[5pt]
 & y \leq w & & &\\[5pt]
 & x+y \geq 1 & & &
\end{array}
\end{equation}
Point $(1,0)$ is the unique solution of problem \eqref{ex: bilev}, while all the infinitely many points $(1-\lambda,\lambda,\lambda)$, with $\lambda \leq 0$, are equilibria of GNEP \eqref{ex: gnep}. In particular, we remark that $(1,0,0)$ is the only solution of GNEP \eqref{ex: gnep} that satisfies assumption (ii) of Theorem \ref{th: solution sets inclusion} (see Figure \ref{fig:ex121} and Figure \ref{fig:ex122}).
\begin{figure}[!ht]
 \begin{minipage}[t]{.48\textwidth}
   \centering
  \includegraphics[width=\textwidth]{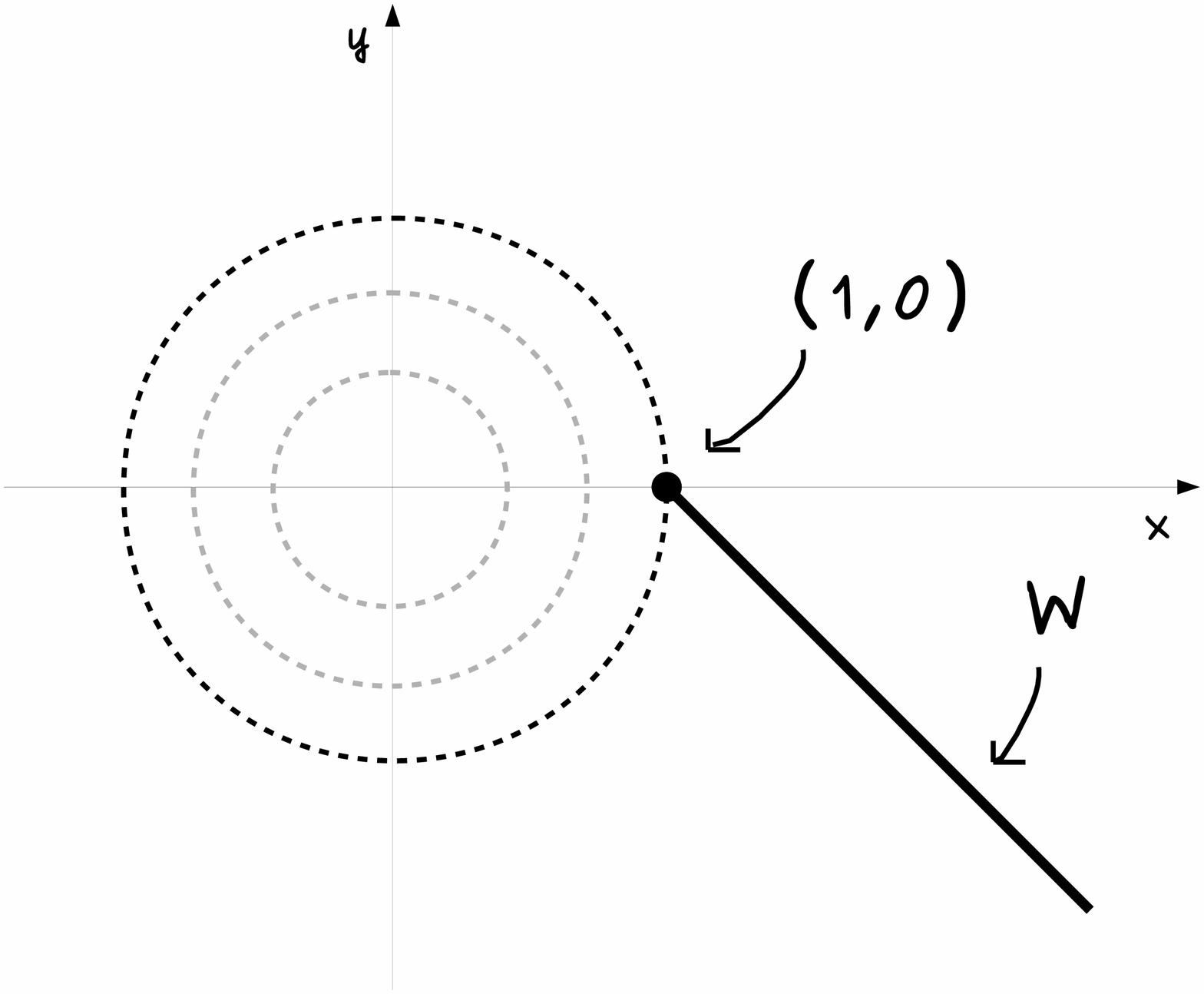}
   \vspace{-.4cm}
   \caption{The feasible set $W$ and the unique solution of SBP \eqref{ex: bilev}}\label{fig:ex121}
 \end{minipage}
 \ \hspace{.01\textwidth} \
 \begin{minipage}[t]{.48\textwidth}
  \centering
   \includegraphics[width=\textwidth]{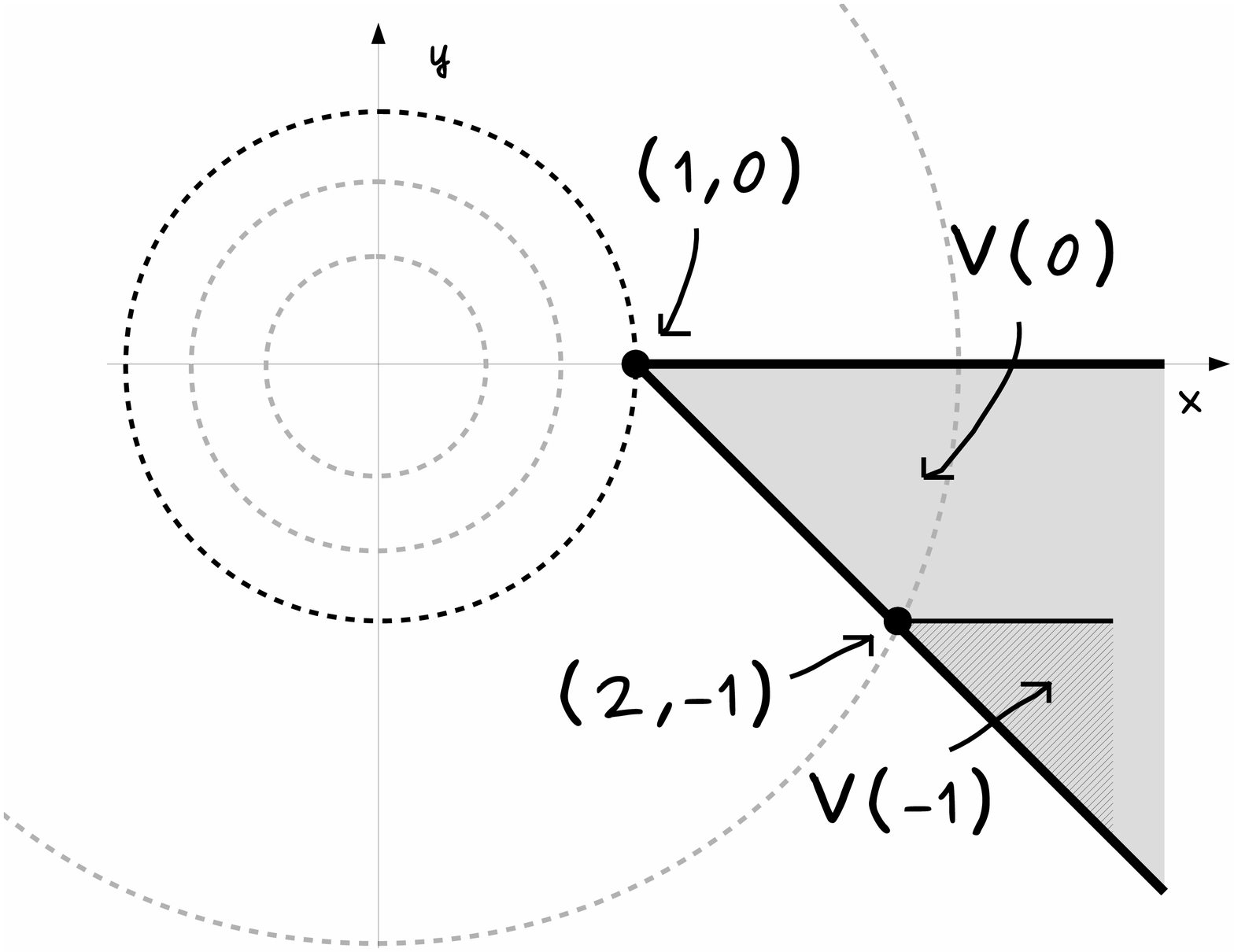}
   \vspace{-.4cm}
   \caption{A sketch of leader's problem in GNEP \eqref{ex: gnep}: the feasible set $V(w)$ and the corresponding solution are depicted for different values of $w$, namely $w=0$ and $w=-1$.}
\label{fig:ex122}
 \end{minipage}
\end{figure}
\hfill \qed \end{example}
It should be remarked (see Example \ref{ex:BPGN}) that the implications in Theorem \ref{th: solution sets inclusion} (ii) can not be reversed: indeed, in general, given a global solution $(x^*,y^*)$ of SBP \eqref{eq: bilevel}, $(x^*,y^*,y^*)$ may not be an equilibrium for GNEP \eqref{eq: gnep}.
\begin{example}\label{ex:BPGN}
\begin{figure}[!ht]
 \begin{minipage}[t]{.48\textwidth}
   \centering
  \includegraphics[width=\textwidth]{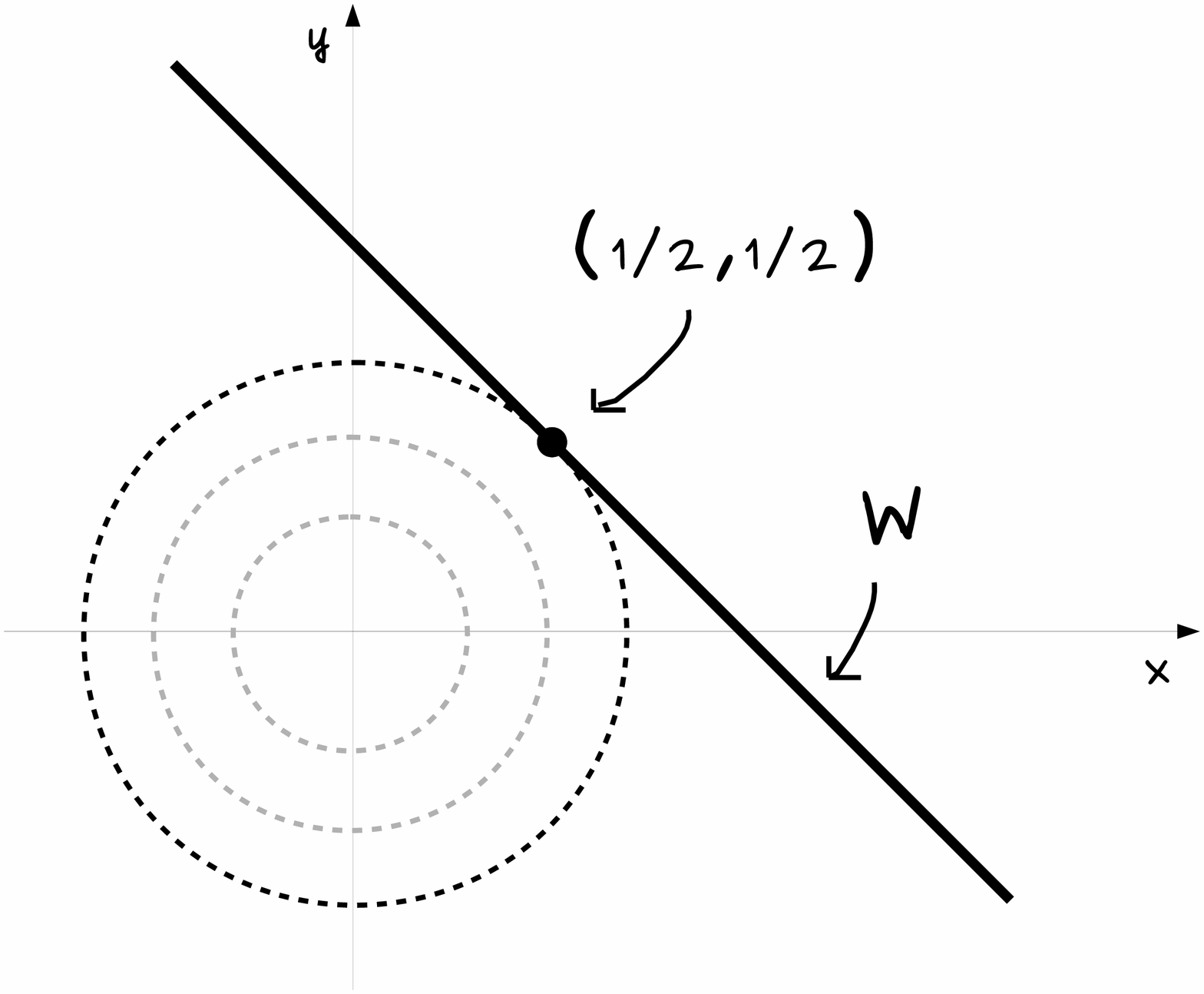}
   \vspace{-.4cm}
   \caption{The feasible set $W$ and the unique solution of SBP \eqref{ex: bilev 2}}\label{fig:ex141}
 \end{minipage}
 \ \hspace{.01\textwidth} \
 \begin{minipage}[t]{.48\textwidth}
  \centering
   \includegraphics[width=\textwidth]{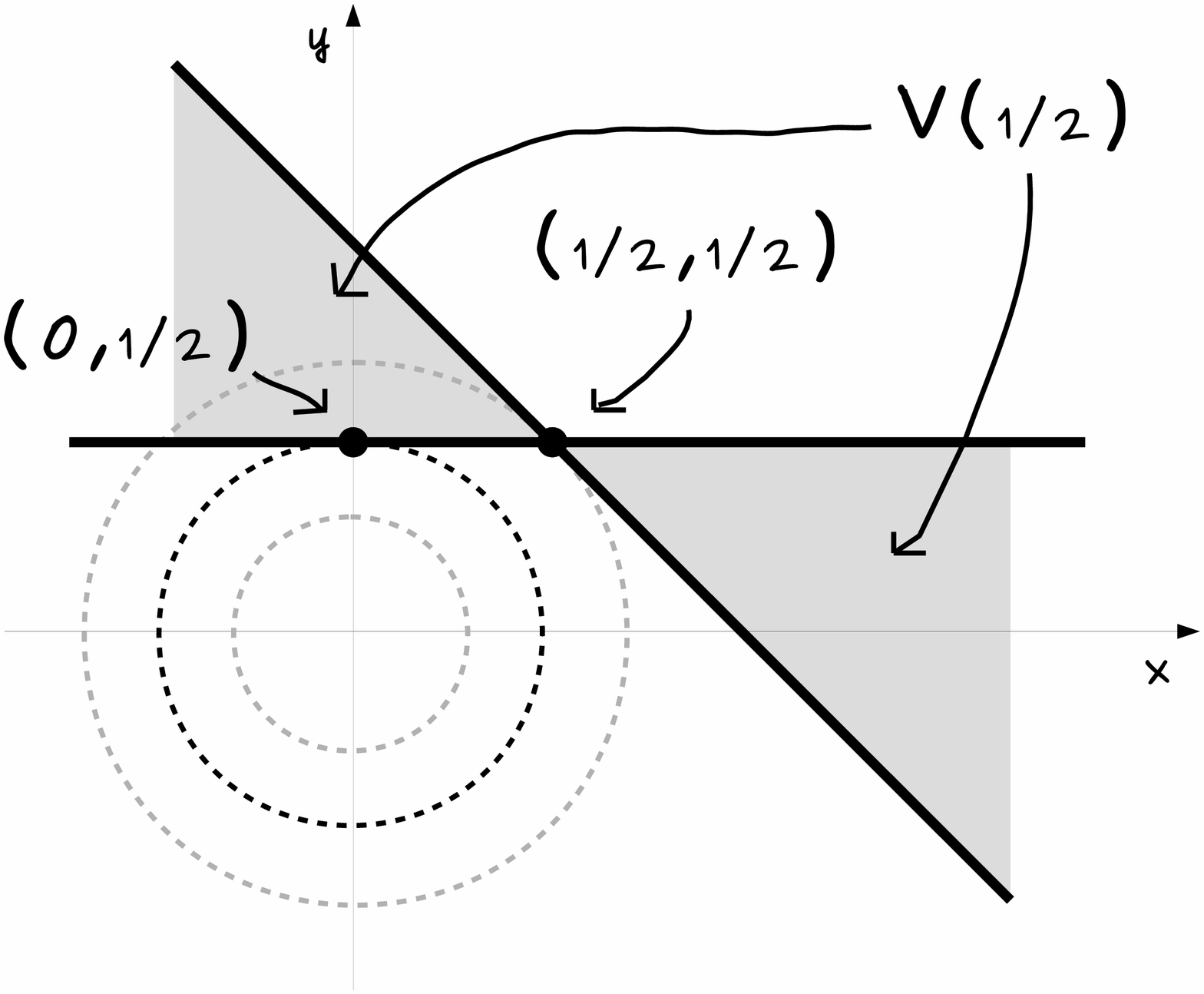}
   \vspace{-.4cm}
   \caption{A sketch of leader's problem in GNEP \eqref{ex: gnep 2}: the feasible set $V(w),$ which turns out to be a superset of $W$, and the corresponding solution are depicted for $w=1/2$.}
\label{fig:ex142}
 \end{minipage}
\end{figure}

Let us consider the following SBP:
\begin{equation}\label{ex: bilev 2}
\begin{array}{cl}
\underset{x,y}{\mbox{minimize}} & x^2+y^2\\
\mbox{s.t.} & y \in S(x),
\end{array}
\end{equation}
where $S(x)$ denotes the solution set of the lower level problem
$$
\begin{array}{cl}
\underset{w}{\mbox{minimize}} & (x+w-1)^2\\
\end{array}
$$
and the corresponding GNEP
\begin{equation}\label{ex: gnep 2}
\begin{array}{clccl}
\underset{x,y}{\mbox{minimize}} & \; x^2+y^2 & \hspace{50pt} & \underset{w}{\mbox{minimize}} & \; (x+w-1)^2.\\
\mbox{s.t.} & (x+y-1)^2\leq (x+w-1)^2 &  & 
\end{array}
\end{equation}
The unique solution of problem \eqref{ex: bilev 2} is $(x^*, y^*) = \left(\frac{1}{2},\frac{1}{2}\right)$. However, the triple $(x^*, y^*, w^*) = \left(\frac{1}{2},\frac{1}{2},\frac{1}{2}\right)$ is not an equilibrium of GNEP \eqref{ex: gnep 2}, since point $(\tilde x, y^*,w^*) = \left(0,\frac{1}{2},\frac{1}{2}\right)$ is feasible for the first player and $\tilde x^2 + (y^*)^2 < (x^*)^2 + (y^*)^2$ (see Figure \ref{fig:ex141} and Figure \ref{fig:ex142}). \hfill \qed
\end{example}
On the other hand, as also observed in \cite{allende2013solving}, strengthening conditions in Theorem \ref{th: solution sets inclusion}, one can define points for which the relation between SBP \eqref{eq: bilevel} and GNEP \eqref{eq: gnep} is stronger than that already established. 
\begin{proposition}\label{pr: easy points}
 Let $(x^*,y^*)$ belong to $W$ and be such that
 \begin{equation}\label{eq: easy bilevel solution optim expl}
  F(x^*,y^*) \leq F(x,y) \; \forall \, (x,y) \in T.
 \end{equation}
 Then
 \begin{enumerate}
  \item[(i)] $(x^*,y^*)$ is a global solution of SBP \eqref{eq: bilevel};
  \item[(ii)] for all $w^* \in U \cap K(x^*)$ such that $(x^*,w^*) \in H(y^*)$, $(x^*, y^*, w^*)$ is a solution of GNEP \eqref{eq: gnep}.
 \end{enumerate}
\end{proposition}
\begin{proof}
 (i) Condition \eqref{eq: easy bilevel solution optim expl} implies relation \eqref{eq: bilevel solution optim expl} since $W \subseteq T$.
 
 (ii) Relations \eqref{eq: gnep solution feasib leader}, \eqref{eq: gnep solution feasib follower} and \eqref{eq: gnep solution optim follower} follow from \eqref{eq: bilevel solution feasib expl} and the fact that $w^* \in U$, $g(x^*,w^*) \leq 0$ and $f(x^*,w^*) = f(x^*,y^*)$. Moreover, \eqref{eq: easy bilevel solution optim expl} implies \eqref{eq: gnep solution optim leader} since $V(w^*) \subseteq T$.
\qed \end{proof}
Points that satisfy conditions \eqref{eq: easy bilevel solution optim expl} can be considered as ``easy'' global solutions of SBP \eqref{eq: bilevel} and, thus, lead also to global solutions of OBP \eqref{ex: origoptim}: such points lie in $W$ but, as for optimality (see relation \eqref{eq: bilevel solution optim expl}), the lower level objective function plays no role for these solutions to be computed. Clearly, if $(x^*,y^*)$ is an ``easy'' solution of SBP \eqref{eq: bilevel}, then, in view of (ii), $(x^*,y^*,y^*)$ is an equilibrium of GNEP \eqref{eq: gnep}. 

In the following example, we present an SBP whose unique solution satisfies the assumptions in Proposition \ref{pr: easy points}.
\begin{example}\label{ex:easy}
Let us consider the following SBP:
\begin{equation}\label{ex: easy bilev}
\begin{array}{cl}
\underset{x,y_1,y_2}{\mbox{minimize}} & x^2+(y_1+y_2)^2 \\
\mbox{s.t.} & x \geq \frac{1}{2} \\[.3em]
            & (y_1,y_2) \in S(x),
\end{array}
\end{equation}
where $S(x)$ denotes the solution set of the lower level problem
$$
\begin{array}{cl}
\underset{w_1,w_2}{\mbox{minimize}} & w_1\\
& x+w_1+w_2 \geq 1\\
& w_1, w_2 \ge 0.
\end{array}
$$
The corresponding GNEP.
\begin{equation}\label{ex: easy gnep}
\begin{array}{clccl}
\underset{x,y_1,y_2}{\mbox{minimize}} & \; x^2+(y_1+y_2)^2 & \hspace{70pt} & \underset{w_1,w_2}{\mbox{minimize}} & \; w_1\\
\mbox{s.t.} & x \geq \frac{1}{2} &  & \mbox{s.t.} & x+w_1+w_2 \geq 1 \\
 & y_1 \leq w_1 & & & w_1, w_2 \geq 0, \\
 & x+y_1+y_2 \geq 1 & & & \\
 & y_1, y_2 \geq 0 & & &
\end{array}
\end{equation}
Clearly, $\left(\frac{1}{2},0,\frac{1}{2}\right)$ is the unique solution of problem \eqref{ex: easy bilev}, while $\left(\frac{1}{2},0,\frac{1}{2},0,\frac{1}{2}\right)$ is an equilibrium of GNEP \eqref{ex: easy gnep}; furthermore, $\left(\frac{1}{2},0,\frac{1}{2}\right)$ satisfies the assumptions in Proposition \ref{pr: easy points}. It is worth pointing out that ``easy'' solution $\left(\frac{1}{2},0,\frac{1}{2}\right)$ can not be calculated by simply minimizing $F(x, y)$ over set $T$: if we did this, indeed, we would obtain multiple solutions, namely any point $\left(\frac{1}{2},y_1,y_2\right)$ such that $y_1 + y_2 = \frac{1}{2}$ and $y_1, \,y_2 \ge 0$. But, among these points, only $\left(\frac{1}{2},0,\frac{1}{2}\right)$ belongs to $W$. Thus, actually, the ``easy'' solutions are not so easy to be calculated! Indeed, although, for optimality, the lower level objective function plays no role for these points to be computed, nonetheless the ``easy'' solutions must belong to the feasible set $W$. \hfill \qed       
\end{example}
Clearly, as stated above, in general, solving GNEP \eqref{eq: gnep} may happen not to lead to a solution of SBP \eqref{eq: bilevel}. However, Theorem \ref{th: solution sets inclusion}, as well as Proposition \ref{pr: easy points}, establish sufficient conditions for an equilibrium of GNEP \eqref{eq: gnep} to provide a global solution of SBP/OBP \eqref{eq: bilevel}/\eqref{ex: origoptim}. Relying on these conditions, with the following Corollary \ref{co: g vanish} and Theorem \ref{th: solution set equiv} we present two significant classes of problems for which one can establish an even deeper connection between global solutions of  SBP/OBP \eqref{eq: bilevel}/\eqref{ex: origoptim} and those of GNEP \eqref{eq: gnep}. 

For example, if the lower level feasible set does not depend on upper level variables $x$, then the requirements of Theorem \ref{th: solution sets inclusion} (ii) are trivially satisfied and the following result, whose proof is omitted, holds.
\begin{corollary}\label{co: g vanish}
Suppose that, at the lower level, the feasible set mapping $U \cap K$ is fixed in $X \cap \text{dom}(U \cap K)$. If $(x^*, y^*, w^*)$ is an equilibrium of GNEP \eqref{eq: gnep}, then $(x^*, y^*)$ is a global solution of SBP \eqref{eq: bilevel}.
\end{corollary}
\begin{remark}\label{rm:Stack}
The class of SBPs in which function $g$ does not depend on the upper variables $x$, in view of the previous corollary, can be solved by addressing GNEP \eqref{eq: gnep}, whenever at least an equilibrium exists. Note that Stackelberg games (see Section \ref{sec:backg}) belong to this category of problems.
\end{remark}
We point out that, as Example \ref{ex:BPGN} shows, also the implications in Corollary \ref{co: g vanish} can not be reversed: given a global solution $(x^*,y^*)$ of SBP \eqref{eq: bilevel}, $(x^*,y^*,y^*)$ may not be an equilibrium for GNEP \eqref{eq: gnep}, even when the lower level feasible set does not depend on upper level variables $x$.
%On the other hand, as the following remark shows, Corollary \ref{co: g vanish} can be effectively exploited in order to obtain a solution of a Stackelberg game from our GNEP model.
%\begin{remark}\label{re: ex2 bis}
%We slightly modify Example \ref{ex:BPGN} by adding to problem \eqref{ex: bilev 2} the private constraint $x \geq \frac{1}{2}$. The unique solution of the SBP is again $(x^*, y^*) = \left(\frac{1}{2},\frac{1}{2}\right)$, and the triple $(x^*, y^*, w^*) = \left(\frac{1}{2},\frac{1}{2},\frac{1}{2}\right)$ is actually the unique equilibrium of the corresponding GNEP model.
%\end{remark}
On the other hand, this is not the case whenever, at the lower level, the solution set mapping $S$ is fixed. Indeed, for this class of problems, the implications in Theorem \ref{th: solution sets inclusion} (ii) can actually be reversed.
\begin{theorem}\label{th: solution set equiv}
Suppose that, at the lower level, the solution set mapping $S$ is fixed in $X \cap \text{dom}(U \cap K)$. The following implications hold:
\begin{enumerate}
\item[(i)]
if $(x^*, y^*, w^*)$ is an equilibrium of GNEP \eqref{eq: gnep}, then $(x^*, y^*)$ is a global solution of SBP \eqref{eq: bilevel};
\item[(ii)]
if $(x^*, y^*)$ is a global solution of SBP \eqref{eq: bilevel}, then, for all $w^* \in U$ such that $g(x^*, w^*) \le 0$ and $f(x^*, w^*) = f(x^*, y^*)$, $(x^*, y^*, w^*)$ is a solution of GNEP \eqref{eq: gnep}.
\end{enumerate}
\end{theorem}
\begin{proof}
In view of relations \eqref{eq: bilevel solution feasib expl}, \eqref{eq: bilevel solution optim expl} and \eqref{eq: gnep solution feasib leader}-\eqref{eq: gnep solution optim follower}, in both cases, it suffices to show that, for every $x \in X \cap \text{dom}(U \cap K)$, $f(x,w^*) = \min_y\{f(x, y): \, y \in K(x) \cap U\}$ and, thus, $W = V(w^*)$.

(i) The claim follows easily observing that $w^* \in S(x^*)$.
 
(ii) Clearly, $y^* \in S(x^*)$ but, since $w^* \in U,$ $g(x^*, w^*) \le 0$ and $f(x^*,y^*) = f(x^*,w^*)$, we also have $w^* \in S(x^*).$  
\qed \end{proof}
\begin{remark}\label{rm: nollx}
Whenever in SBP \eqref{eq: bilevel}, at the lower level, the whole dependence on $x$ is dropped, the solution set mapping $S$ is obviously fixed. Thus, pure hierarchical program \eqref{eq: hier}
belonging to this category of problems, can equivalently be reformulated as the following simple GNEP in which the coupling between leader's and follower's problems occurs only at the leader's feasible set level:
\begin{equation}\label{eq: hier gnep}
\begin{array}{clccl}
\underset{x,y}{\mbox{minimize}} & \; F(x,y) & \hspace{70pt} & \underset{w}{\mbox{minimize}} & \; f(w)\\
\mbox{s.t.} & (x,y) \in X \times U &  & \mbox{s.t.} & w \in U\\[5pt]
 & f(y) \leq f(w) & & & g(w)\leq 0.\\[5pt]
 & g(y) \leq 0 & & &
\end{array}
\end{equation}
\end{remark}
Here we consider a particularly interesting example of SBP with a fixed lower level solution set mapping.
\begin{example}\label{ex:dempedutta}
(see \cite{dempe2012bilevel}) Let us consider the following SBP:
\begin{equation}\label{ex: dempedutta}
\begin{array}{cl}
\underset{x,y}{\mbox{minimize}} & (x-1)^2+y^2 \\
\mbox{s.t.} & y \in S(x),
\end{array}
\end{equation}
where $S(x)$ denotes the solution set of the lower level problem
$$
\begin{array}{cl}
\underset{w}{\mbox{minimize}} & x^2 \, w\\
& w^2 \leq 0.
\end{array}
$$
Note that the unique solution of \eqref{ex: dempedutta} is $\left(1,0\right)$.

Interestingly, as shown in \cite{dempe2012bilevel}, solving the MPCC reformulation of SBP \eqref{ex: dempedutta} invariably leads to point $(0,0),$ which is not the solution of the original problem. In this case, the MPCC reformulation fails to identify the set of solutions of the SBP, due to the lack of regularity (Slater's condition) in the lower level feasible set (see \cite{dempe2012bilevel}). Our GNEP, instead, in view of the previous result, effectively provides the unique solution of SBP \eqref{ex: dempedutta}.
Indeed, it is worth remarking that, in order to address SBP/OBP \eqref{eq: bilevel}/\eqref{ex: origoptim} by means of GNEP \eqref{eq: gnep}, we do not need any convexity or regularity preliminary assumption. \hfill \qed 
\end{example}

\subsection{Strong local solutions}\label{sec:locstr}

SBPs are inherently nonconvex (see \cite{dempe2007new}), so that multiple local optimal solutions may occur. We say that $(x^*,y^*)$ is a strong local solution of SBP \eqref{eq: bilevel} if $(x^*,y^*) \in W$ and there exists a neighborhood $N^* \in {\cal N}(x^*)$ of $x^*$ such that
\begin{equation}
F(x^*,y^*) \leq F(x,y) \; \forall \, (x,y) \in W \cap (N^* \times \mathbb R^{n_2}).
 \label{eq: bilevel local solution optim expl}
\end{equation}
Of course, global solutions are strong local solutions of SBP \eqref{eq: bilevel}. 
As the following example clearly shows, even if the lower level problem is linear and the upper level objective function is strongly convex, the resulting SBP may be nonconvex. Moreover, in this case, strong local solutions that are not global occur.
\begin{example}\label{ex:nonconvex}
Let us consider the following SBP:
\begin{equation}\label{ex: nonconvex bilev}
\begin{array}{cl}
\underset{x,y}{\mbox{minimize}} & x^2+y^2 \\
\mbox{s.t.} & -1 \leq x \leq 1 \\
& y \in S(x),
\end{array}
\end{equation}
where $S(x)$ denotes the solution set of the lower level problem
$$
\begin{array}{cl}
\underset{w}{\mbox{minimize}} & -w\\
& 2x+w \leq 2\\
& 0 \leq w \leq 1.
\end{array}
$$
Point $\left(\frac{4}{5},\frac{2}{5}\right)$ is the global solution of problem \eqref{ex: nonconvex bilev}, while $\left(0,1\right)$ is a strong local solution of SBP \eqref{ex: nonconvex bilev} that is not global. \hfill \qed
\end{example}
We point out that strong local solutions are obviously local solutions for SBP \eqref{eq: bilevel}. The converse, in general, is not true, see the following example.  
\begin{example}\label{ex:mordu}
(see \cite{dempe2012sensitivity}) Consider the following SBP
\begin{equation}\label{ex: nonconvex bilevmordu}
\begin{array}{cl}
\underset{x,y}{\mbox{minimize}} & x \\
\mbox{s.t.} & -1 \leq x \leq 1 \\
& y \in S(x),
\end{array}
\end{equation}
where $S(x)$ denotes the solution set of the lower level problem
$$
\begin{array}{cl}
\underset{w}{\mbox{minimize}} & xw\\
& 0 \leq w \leq 1.
\end{array}
$$
Point $(0,0)$ is a local solution that is not strong. Moreover, notice that the unique global minimum $(-1,1)$ is an ``easy'' global solution of SBP \eqref{ex: nonconvex bilevmordu} (see Proposition \ref{pr: easy points}). \hfill \qed 
\end{example}
Strong local solutions can be considered as ``asymmetric'' local solutions, since, in some sense, variables $x$ play there a more important role. Interestingly, any strong local solution of SBP \eqref{eq: bilevel}, which is precisely what we seek for, leads to a local solution of OBP \eqref{ex: origoptim}, unlike generic local solutions of SBP \eqref{eq: bilevel} (see \cite{dempe2012sensitivity}).
\begin{proposition}\label{th: strlocoptim}
Let $(x^*,y^*) \in W$ be a strong local solution of SBP \eqref{eq: bilevel}. Then $x^*$ is a local solution of OBP \eqref{ex: origoptim}.  
\end{proposition}
\begin{proof}
Since $(x^*,y^*) \in W$ and there exists a neighborhood $N^* \in {\cal N}(x^*)$ of $x^*$ such that $F(x^*,y^*) \leq F(x,y) \; \forall \, (x,y) \in W \cap (N^* \times \mathbb R^{n_2})$, we have $\min_y\{F(x^*, y): \, y \in S(x^*)\} = F(x^*,y^*) \le \min_y\{F(x, y): \, y \in S(x)\} \; \forall \, x \in X \cap N^*$. Hence, $x^*$ is a local solution of \eqref{ex: origoptim}. 
\qed \end{proof}
We are now in a position to restate Theorem \ref{th: solution sets inclusion} (ii) in a local sense. Preliminarily, let $I(x,y) \triangleq \left\{i \in \{1, \ldots, m\}: \; g_i(x,y) = 0\right\}$ be the active index set for constraints $g$ at $(x, y)$.
\begin{theorem}\label{th: local solution sets inclusion}
Let $(x^*, y^*, w^*)$ be an equilibrium of GNEP \eqref{eq: gnep}. If, for every $i \in I(x^*, w^*)$, there exists a neighborhood $N^* \in {\cal N}(x^*)$ of $x^*$ with $g_i(x,w^*)\leq 0$ for all $x \in N^*$ such that there exists $y$ with $(x,y) \in W$ and $F(x,y) \leq F(x^*,y^*)$, then $(x^*, y^*)$ is a strong local solution of SBP \eqref{eq: bilevel}.
\end{theorem}
\begin{proof}
Since $(x^*, y^*, w^*)$ is an equilibrium of GNEP \eqref{eq: gnep}, it satisfies relations \eqref{eq: gnep solution feasib leader}-\eqref{eq: gnep solution optim follower}. Our aim is to show that \eqref{eq: bilevel solution feasib expl} and \eqref{eq: bilevel local solution optim expl} hold at $(x^*, y^*)$.
 
As done in the proof of Theorem \ref{th: solution sets inclusion}, we observe that \eqref{eq: gnep solution feasib leader}, \eqref{eq: gnep solution feasib follower} and \eqref{eq: gnep solution optim follower} together imply that $(x^*, y^*)$ satisfies \eqref{eq: bilevel solution feasib expl}: thus, $(x^*,y^*)$ is feasible for SBP \eqref{eq: bilevel}.

We recall that, by \eqref{eq: gnep solution feasib follower}, we have $g(x^*, w^*) \le 0$; let, without loss of generality, $N^*$ be such that $g_j(x, w^*) \le 0$ for all $x \in N^*$ and for every $j \notin I(x^*, w^*)$. 

For any couple $(\bar x, \bar y)$ in $W \cap (N^* \times \Re^{n_2}) \cap {\cal L}^*$ (for the definition of sets ${\cal L}^*$ and $({\cal L}^*)^c$, see \eqref{eq:L} and \eqref{eq:Lc}) we have, by assumptions, $g_i(\bar x,w^*)\leq 0$ for every $i \in I(x^*,w^*)$. Therefore, since we have also $g_j(\bar x, w^*) \le 0$ for every $j \notin I(x^*,w^*)$, in view of \eqref{eq: gnep solution feasib follower}, we get $w^* \in U \cap K(\bar x)$.
Inclusions $(\bar x, \bar y) \in W$ and $w^* \in U \cap K(\bar x)$ entail $(\bar x, \bar y) \in V(w^*)$ and, in turn,
\begin{equation}\label{eq:fininc}
W \cap (N^* \times \Re^{n_2}) \cap {\cal L}^* \subseteq V(w^*).
\end{equation}
Thanks to \eqref{eq: gnep solution optim leader} and \eqref{eq:fininc}, and noting that for every $(x,y) \in W \cap (N^* \times \Re^{n_2}) \cap ({\cal L}^*)^c$ we have $F(x,y) > F(x^*,y^*)$, \eqref{eq: bilevel local solution optim expl} holds at $(x^*, y^*)$. Hence, $(x^*, y^*)$ is a strong local solution of SBP \eqref{eq: bilevel}. 
\qed \end{proof}
We remark that, while Example \ref{ex:BPGN} shows that a strong (local) solution of SBP may not lead to an equilibrium of the corresponding GNEP, if conditions in Theorem \ref{th: local solution sets inclusion} are satisfied, an equilibrium of GNEP \eqref{eq: gnep} always provides us with a strong (local) solution of SBP \eqref{eq: bilevel}.

\begin{example}\label{ex:nonconvex2}
Let us consider again the problem in Example \ref{ex:nonconvex}; the corresponding GNEP is the following:
\begin{equation}\label{ex: nonconvex gnep}
\begin{array}{clccl}
\underset{x,y}{\mbox{minimize}} & \; x^2+y^2 & \hspace{50pt} & \underset{w}{\mbox{minimize}} & \; -w\\
\mbox{s.t.} & -1 \leq x \leq 1 &  & & 2x+w \leq 2 \\
& y \geq w && & 0 \leq w \leq 1. \\
& 2x+y \leq 2 &&& \\
& 0 \leq y \leq 1 &&&
\end{array}
\end{equation}
Point $\left(0,1,1\right)$ is an equilibrium of the convex GNEP \eqref{ex: nonconvex gnep}. Moreover, it trivially satisfies assumptions of Theorem \ref{th: local solution sets inclusion}, since constraint $2x+w \leq 2$ is not active at $\left(0,1\right)$. \hfill \qed
\end{example}

\section{Applications in economics}\label{sec: application}

Let us consider a market with two firms, each acting as a player. Firm 1 produces quantities $q^1 \in \Re^{n_1}$ of some goods, while firm 2 produces quantities $q^2 \in \Re^{n_2}$ of other goods. Given private technological constraints $X^\nu$ on the production level, each player $\nu = 1, 2$ sets $q^\nu$ in order to maximize its own profit
$$
\Pi_\nu (q^1,q^2) := p^\nu(q^1,q^2)\trt q^\nu - c_\nu(q^1,q^2),
$$
where $p^\nu$ and $c_\nu$ are inverse demand and cost functions, respectively.
%where $p^\nu$ is the inverse demand function of the goods and $c_\nu$ is the cost function.
We assume sets $X^1$ and $X^2$ to be convex, compact and nonempty, and functions $\Pi_1$ and $\Pi_2$ to be continuously differentiable and concave with respect to $(q^1,q^2)$ and $q^2$, respectively.
In this setting, two different classical perspectives can be considered.
\begin{description}
 \item[Horizontal model:] both players decide their strategies $q^\nu$ simultaneously; we assume that the players act rationally and have complete information, and there is no explicit collusion; we model this case as a ``standard'' GNEP;
 
 \item[Vertical model:] player 1 can anticipate player 2 by setting its variables $q^1$ for first; we model this case as an SBP.
\end{description}
We illustrate that, in order to model this system, one can also rely on our new GNEP \eqref{eq: gnep}, which in some sense lies in between the horizontal and the vertical models. We call our GNEP {\bf uneven horizontal model}.

In the following subsections, considering different instances of the described framework, we highlight the connections between the three models.

\subsection{$\Pi_2$ not depending on $q^1$}\label{subsec: application hierarchical}

Assume that $\Pi_2$ does not depend on $q^1$.
From an horizontal point of view, the system can be modeled by resorting to the following ``standard'' GNEP:
\begin{equation*}
\begin{array}{clccl}
\underset{q^1}{\mbox{maximize}} & \; \Pi_1(q^1,q^2) & \hspace{70pt} & \underset{q^2}{\mbox{maximize}} & \; \Pi_2(q^2)\\
\mbox{s.t.} & q^1 \in X^1 &  & \mbox{s.t.} & q^2 \in X^2.
\end{array}
\end{equation*}
In a vertical framework, one can rely to the classical (hierarchical) optimization problem
\begin{equation*}
\begin{array}{cl}
\underset{q^1, q^2}{\mbox{maximize}} & \Pi_1(q^1,q^2) \\
\mbox{s.t.} & q^1 \in X^1 \\[5pt]
& q^2 \in S,
\end{array}
\end{equation*}
where $S$ denotes the solution set of the lower level problem
$$
\begin{array}{cl}
\underset{w^2}{\mbox{maximize}} & \Pi_2(w^2)\\
\mbox{s.t.} & w^2 \in X^2.
\end{array}
$$
Finally, a new intermediate perspective can be given by the uneven horizontal GNEP model  
\begin{equation*}
\begin{array}{clccl}
\underset{q^1,q^2}{\mbox{maximize}} & \; \Pi_1(q^1,q^2) & \hspace{70pt} & \underset{w^2}{\mbox{maximize}} & \; \Pi_2(w^2)\\
\mbox{s.t.} & q^1 \in X^1, \quad q^2 \in X^2 &  & \mbox{s.t.} & w^2 \in X^2.\\[5pt]
& \Pi_2(q^2) \geq \Pi_2(w^2) & & & 
\end{array}
\end{equation*}

\noindent
Let us introduce the following sets of values:
\begin{description}
 \item[$\Pi_1^{Horizontal}$] is the range of values of $\Pi_1$ with respect to the solution set of the horizontal model: given an equilibrium $(\widetilde q^1, \widetilde q^2)$ of the horizontal model, we have $\Pi_1(\widetilde q^1, \widetilde q^2) \in \Pi_1^{Horizontal}$;
 \item[$\Pi_1^{Uneven}$] is the range of values of $\Pi_1$ with respect to the solution set of the uneven horizontal model:
 given an equilibrium $(\bar q^1, \bar q^2, \bar w^2)$ of the uneven horizontal model, we have $\Pi_1(\bar q^1, \bar q^2) \in \Pi_1^{Uneven}$;
 \item[$\Pi_1^{Vertical}$] is the optimal value of the vertical model.
\end{description}

\noindent
By assumptions, $\Pi_1^{Horizontal}$ (see \cite{FacchPangBk}) is compact and nonempty, while $\Pi_1^{Uneven}$ and $\Pi_1^{Vertical}$ are singletons. Then, the connections between the three modelistic perspectives can be expressed by the following straightforward relations (see also Theorem \ref{th: solution set equiv}):
%By assumptions, $\Pi_1^{Horizontal}$ (see \cite{FacchPangBk}), $\Pi_1^{Uneven}$ and $\Pi_1^{Vertical}$ are compact and nonempty. In this case $\Pi_1^{Uneven}$ is a singleton.
%Moreover, by referring to Theorem \ref{th: solution set equiv}, we get the following relations:
$$
\max \{ \Pi_1^{Horizontal} \} = \Pi_1^{Uneven} = \Pi_1^{Vertical}.
$$
\begin{remark}\label{rm:algohier}
It should be remarked that $\Pi_1^{Uneven}$ can be computed by simply finding the optimal value $\Pi_2^*$ of the follower's problem and then addressing the optimization problem
\begin{equation*}
\begin{array}{clccl}
\underset{q^1,q^2}{\mbox{maximize}} & \; \Pi_1(q^1,q^2) & & &\\
\mbox{s.t.} & q^1 \in X^1, \quad q^2 \in X^2 &  &  &\\[5pt]
& \Pi_2(q^2) \geq \Pi_2^*. & & & 
\end{array}
\end{equation*}
\end{remark}

\subsection{$\Pi_2$ not depending on $q^1$ and players sharing a budget constraint}\label{subsec: application general 1}

In the same setting of subsection \ref{subsec: application hierarchical}, let players also share a common resource. Thus, for every player, we consider the additional budget constraint $a_1(q^1) + a_2(q^2) \leq b$, where convex function $a_\nu$ ($\nu = 1, 2$) indicates the resource consumption to produce quantities $q^\nu$ and scalar $b>0$ is the amount of resource available in the market.
We assume set $\{q^1 \in X^1, \, q^2 \in X^2 \, : \, a_1(q^1) + a_2(q^2) \leq b \}$ to be nonempty.
In an horizontal framework we have:
\begin{equation*}
\begin{array}{clccl}
\underset{q^1}{\mbox{maximize}} & \; \Pi_1(q^1,q^2) & \hspace{70pt} & \underset{q^2}{\mbox{maximize}} & \; \Pi_2(q^2)\\
\mbox{s.t.} & q^1 \in X^1 &  & \mbox{s.t.} & q^2 \in X^2\\[5pt]
& a_1(q^1) + a_2(q^2) \leq b & & & a_1(q^1) + a_2(q^2) \leq b,
\end{array}
\end{equation*}
as for the vertical model we get:
\begin{equation*}
\begin{array}{cl}
\underset{q^1, q^2}{\mbox{maximize}} & \Pi_1(q^1,q^2) \\
\mbox{s.t.} & q^1 \in X^1 \\[5pt]
& q^2 \in S(q^1),
\end{array}
\end{equation*}
where $S(q^1)$ denotes the solution set of the lower level problem
$$
\begin{array}{cl}
\underset{w^2}{\mbox{maximize}} & \Pi_2(w^2)\\
\mbox{s.t.} & w^2 \in X^2 \\[5pt]
& a_1(q^1) + a_2(w^2) \leq b.
\end{array}
$$
In the uneven horizontal vision, we have:
\begin{equation*}
\begin{array}{clccl}
\underset{q^1,q^2}{\mbox{maximize}} & \; \Pi_1(q^1,q^2) & \hspace{70pt} & \underset{w^2}{\mbox{maximize}} & \; \Pi_2(w^2)\\
\mbox{s.t.} & q^1 \in X^1, \quad q^2 \in X^2 &  & \mbox{s.t.} & w^2 \in X^2 \\[5pt]
& a_1(q^1) + a_2(q^2) \leq b & & & a_1(q^1) + a_2(w^2) \leq b. \\[5pt]
& \Pi_2(q^2) \geq \Pi_2(w^2) & & &
\end{array}
\end{equation*}

\noindent
In order to point out the relations between the models, in this case it is useful to resort to the resource-directed parameterization introduced (for jointly convex GNEPs) in \cite{NabTsengFuk09} and in \cite{facchinei2011computation}.
Let $b_1 \in B \triangleq \big\{b_1 \in \Re: 0 \leq b_1 \leq b, \, \{q^1 \in X^1 \, : \, a_1(q^1) \leq b_1 \} \neq \emptyset$ and $\{q^2 \in X^2 \, : \, a_2(q^2) \leq b-b_1 \} \neq \emptyset \big\}$ be the amount of resource given to player 1; on the other hand, $b - b_1 \geq 0$ turns out to be the amount of resource available to player 2. We get the following parameterized version of the horizontal model:
\begin{equation*}
\begin{array}{clccl}
\underset{q^1}{\mbox{maximize}} & \; \Pi_1(q^1,q^2) & \hspace{70pt} & \underset{q^2}{\mbox{maximize}} & \; \Pi_2(q^2)\\
\mbox{s.t.} & q^1 \in X^1 &  & \mbox{s.t.} & q^2 \in X^2\\[5pt]
& a_1(q^1) \leq b_1 & & & a_2(q^2) \leq b - b_1.
\end{array}
\end{equation*}
As parameterized vertical model we have
\begin{equation*}
\begin{array}{cl}
\underset{q^1, q^2}{\mbox{maximize}} & \Pi_1(q^1,q^2) \\
\mbox{s.t.} & q^1 \in X^1 \\[5pt]
& a_1(q^1) \leq b_1 \\[5pt]
& q^2 \in S_{b_1},
\end{array}
\end{equation*}
where $S_{b_1}$ denotes the solution set of the lower level problem
$$
\begin{array}{cl}
\underset{w^2}{\mbox{maximize}} & \Pi_2(w^2)\\
\mbox{s.t.} & w^2 \in X^2 \\[5pt]
& a_2(w^2) \leq b-b_1.
\end{array}
$$
And the corresponding parameterized uneven horizontal version is
\begin{equation*}
\begin{array}{clccl}
\underset{q^1,q^2}{\mbox{maximize}} & \; \Pi_1(q^1,q^2) & \hspace{70pt} & \underset{w^2}{\mbox{maximize}} & \; \Pi_2(w^2)\\
\mbox{s.t.} & q^1 \in X^1, \quad q^2 \in X^2 &  & \mbox{s.t.} & w^2 \in X^2 \\[5pt]
& a_1(q^1) \leq b_1 & & & a_2(w^2) \leq b - b_1. \\[5pt]
& a_2(q^2) \leq b - b_1 & & & \\[5pt]
& \Pi_2(q^2) \geq \Pi_2(w^2) & & &
\end{array}
\end{equation*}

\noindent
As done for sets $\Pi_1^{Horizontal}$, $\Pi_1^{Uneven}$ and $\Pi_1^{Vertical}$ (see subsection \ref{subsec: application hierarchical}), let us define the following sets of values:
\begin{description}
 \item[$\Pi_1^{Horizontal}(b_1)$] is the range of values of $\Pi_1$ with respect to the solution set of the parameterized horizontal model:
 given an equilibrium $(\widetilde q^1, \widetilde q^2)$ of the parameterized horizontal model with $b_1 \in B$, we have $\Pi_1(\widetilde q^1, \widetilde q^2) \in \Pi_1^{Horizontal}(b_1)$;
 \item[$\Pi_1^{Uneven}(b_1)$] is the range of values of $\Pi_1$ with respect to the solution set of the parameterized uneven horizontal model:  given an equilibrium $(\bar q^1, \bar q^2, \bar w^2)$ of the parameterized uneven horizontal model with $b_1 \in B$, we have $\Pi_1(\bar q^1, \bar q^2) \in \Pi_1^{Uneven}(b_1)$;
 \item[$\Pi_1^{Vertical}(b_1)$] is the optimal value of the parameterized vertical model with $b_1 \in B$.
\end{description}
Similarly to what observed in subsection \ref{subsec: application hierarchical}, by assumptions,
%$\Pi_1^{Vertical}$ is compact and nonempty, as well as 
$\Pi_1^{Horizontal}(b_1)$ is compact and nonempty, while $\Pi_1^{Uneven}(b_1)$ and $\Pi_1^{Vertical}(b_1)$ are singletons for every $b_1 \in B$.
%In this case $\Pi_1^{Uneven}(b_1)$ is a singleton.
In this case $\Pi_1^{Horizontal}$ is nonempty since at least a variational equilibrium exists, see \cite{FacchKanSu}. As for $\Pi_1^{Uneven}$, let us assume that an equilibrium of the uneven horizontal model exists, thus making $\Pi_1^{Uneven}$ nonempty. We observe that, by relying for example on Ichiishi's theorem, the latter assumption holds under mild conditions, see, again, \cite{FacchKanSu} (and also Remark \ref{rm:algofnond}); we do not go into details, since this aspect is immaterial to our analysis. Finally, as in the previous case, $\Pi_1^{Vertical}$ is a singleton. 

\noindent
For all $b_1 \in B$, we have
\begin{equation}\label{eq:rel1}
\max \{ \Pi_1^{Horizontal}(b_1) \} = \Pi_1^{Uneven}(b_1) = \Pi_1^{Vertical}(b_1).
\end{equation}
Furthermore, by Theorem \ref{th: solution sets inclusion}, we get
\begin{equation}\label{eq:relpar}
\sup \{ \Pi_1^{Horizontal} \} \leq \sup \{ \Pi_1^{Uneven} \} \leq \Pi_1^{Vertical}.
\end{equation}
Interestingly, relations \eqref{eq:rel1} and \eqref{eq:relpar} can be linked to each other according to the following Propositions \ref{pr: SBP parameterized} and \ref{pr: SBP parameterized 2}.

\begin{proposition}\label{pr: SBP parameterized}
 $$
 \displaystyle \bigcup_{b_1 \in B} \Pi_1^{Vertical}(b_1) \ni \Pi_1^{Vertical}.
 $$
\end{proposition}
\begin{proof}
 Let $(\widehat q^1, \widehat q^2)$ be a solution of the vertical model. With $\widehat b_1 \triangleq a_1(\widehat q^1) \in B,$ we have $S(\widehat q^1) = S_{\widehat b_1}$. Then, in turn, since $(\widehat q^1, \widehat q^2)$ is optimal for the parameterized vertical model with $b_1 = \widehat b_1$, the thesis follows.
\end{proof}
In view of the previous result and since $\Pi_1^{Uneven}(b_1) = \Pi_1^{Vertical}(b_1)$, we also have \begin{equation}\label{eq: application 2 inclusion}
 \displaystyle \bigcup_{b_1 \in B} \Pi_1^{Uneven}(b_1) \ni \Pi_1^{Vertical}.
 \end{equation}
\begin{proposition}\label{pr: SBP parameterized 2}
 If, for every solution $(\widehat q^1, \widehat q^2)$ of the parameterized horizontal model for $b_1 = \widehat b_1 \in B$, $a_1(\widehat q^1) = \widehat b_1$ and $a_2(\widehat q^2) = b-\widehat b_1$, then
 $$
 \displaystyle \sup_{b_1 \in B} \Pi_1^{Uneven}(b_1) = \Pi_1^{Vertical}.
 $$
\end{proposition}
\begin{proof}
 Thanks to \cite[Theorem 3.6]{NabTsengFuk09}, we have $\bigcup_{b_1 \in B} \Pi_1^{Horizontal}(b_1) = \Pi_1^{Horizontal}$, and, in turn, 
 $$
 \begin{array}{rcl}
\displaystyle \sup_{b_1 \in B} \max \{ \Pi_1^{Horizontal}(b_1) \} & = &  \sup \left\{\displaystyle \bigcup_{b_1 \in B} \max \{ \Pi_1^{Horizontal}(b_1) \} \right\}\\[20pt]
 & \le & \sup \left\{\displaystyle \bigcup_{b_1 \in B} \Pi_1^{Horizontal}(b_1) \right\} = \sup \{ \Pi_1^{Horizontal} \}.
 \end{array}
 $$
 Therefore, in view of \eqref{eq:rel1} and \eqref{eq:relpar},
 $$
 \displaystyle \sup_{b_1 \in B} \Pi_1^{Uneven}(b_1) \leq \Pi_1^{Vertical},
 $$
 and the thesis follows by \eqref{eq: application 2 inclusion}.
\end{proof}
We remark that assumptions in Proposition \ref{pr: SBP parameterized 2} simply require that, for every choice of $b_1 \in B$, the common resource is entirely consumed by the players.

\begin{remark}\label{rm:algofnond}
As for the parameterized uneven horizontal game, $\Pi_1^{Uneven}(b_1)$ can be computed, for every fixed $b_1 \in B,$ by relying again on the very simple approach described in Remark \ref{rm:algohier}. Furthermore, one can also calculate a single value belonging to $\Pi_1^{Uneven}$ by resorting to a similar procedure as the one just illustrated (but, in general, with more than one leader/follower optimization). It can be proved that this alternating optimization approach converges to an equilibrium of the uneven horizontal game under mild standard conditions. For the sake of brevity and since this kind of study goes out of the scope of this work, we do not go into details.  
\end{remark}

\subsection{$\Pi_2$ depending on both $q^1$ and $q^2$, and players sharing a budget constraint}\label{subsec: application general 2}

Let us consider the general case in which $\Pi_2$ depends also on $q^1$ and players share a common budget constraint as in subsection \ref{subsec: application general 1}.
Both the horizontal
\begin{equation*}
\begin{array}{clccl}
\underset{q^1}{\mbox{maximize}} & \; \Pi_1(q^1,q^2) & \hspace{70pt} & \underset{q^2}{\mbox{maximize}} & \; \Pi_2(q^1,q^2)\\
\mbox{s.t.} & q^1 \in X^1 &  & \mbox{s.t.} & q^2 \in X^2 \\[5pt]
& a_1(q^1) + a_2(q^2) \leq b & & & a_1(q^1) + a_2(q^2) \leq b, \end{array}
\end{equation*}
and the uneven horizontal
\begin{equation*}
\begin{array}{clccl}
\underset{q^1,q^2}{\mbox{maximize}} & \; \Pi_1(q^1,q^2) & \hspace{70pt} & \underset{w^2}{\mbox{maximize}} & \; \Pi_2(q^1,w^2)\\
\mbox{s.t.} & q^1 \in X^1, \quad q^2 \in X^2 &  & \mbox{s.t.} & w^2 \in X^2 \\[5pt]
& a_1(q^1) + a_2(q^2) \leq b & & & a_1(q^1) + a_2(w^2) \leq b, \\[5pt]
& \Pi_2(q^1,q^2) \geq \Pi_2(q^1,w^2) & & &
\end{array}
\end{equation*}
models are GNEPs. 
Clearly, in order to establish connections between the vertical
\begin{equation*}
\begin{array}{cl}
\underset{q^1, q^2}{\mbox{maximize}} & \Pi_1(q^1,q^2) \\
\mbox{s.t.} & q^1 \in X^1 \\[5pt]
& q^2 \in S(q^1),
\end{array}
\end{equation*}
where $S(q^1)$ denotes the solution set of the lower level problem
$$
\begin{array}{cl}
\underset{w^2}{\mbox{maximize}} & \Pi_2(q^1,w^2)\\
\mbox{s.t.} & w^2 \in X^2 \\[5pt]
& a_1(q^1) + a_2(w^2) \leq b,
\end{array}
$$
and the uneven horizontal models, one can resort to Theorems \ref{th: solution sets inclusion} and \ref{th: local solution sets inclusion}, or, if there is no budget (shared) constraint, to Corollary \ref{co: g vanish}.
In any case (see the definitions introduced in subsection \ref{subsec: application hierarchical}), we have
$$
\sup \{ \Pi_1^{Horizontal} \} \leq \sup \{ \Pi_1^{Uneven} \} \leq \Pi_1^{Vertical}.
$$
Let us consider now the interesting case in which one wants to design the market in order to easily compute a solution of the vertical model. For this to be done, one can exploit Proposition \ref{pr: easy points}: letting $(\widehat q^1, \widehat q^2, \widehat w^2)$ be a solution of the following (jointly convex) GNEP
\begin{equation*}
\begin{array}{clccl}
\underset{q^1,q^2}{\mbox{maximize}} & \; \Pi_1(q^1,q^2) & \hspace{70pt} & \underset{w^2}{\mbox{maximize}} & \; \Pi_2(q^1,w^2)\\
\mbox{s.t.} & q^1 \in X^1, \quad q^2 \in X^2 &  & \mbox{s.t.} & w^2 \in X^2 \\[5pt]
& a_1(q^1) + a_2(q^2) \leq b & & & a_1(q^1) + a_2(w^2) \leq b, 
\end{array}
\end{equation*}
such that $\Pi_2(\widehat q^1,\widehat q^2) \geq \Pi_2(\widehat q^1,\widehat w^2)$, $(\widehat q^1, \widehat q^2)$ is an easy solution (see Proposition \ref{pr: easy points}) of the vertical model.
In the same spirit, an alternative and easier way to compute such solutions makes use of variational inequalities: indeed, $(\widehat q^1, \widehat q^2) \in T = \{(q^1,q^2) \in X^1 \times X^2 \, : \, a_1(q^1) + a_2(q^2) \leq b \}$ such that
$$
\nabla \Pi_1 (\widehat q^1, \widehat q^2)\trt \left((q^1,q^2)-(\widehat q^1, \widehat q^2)\right) \leq 0, \qquad
\nabla_{q^2} \Pi_2 (\widehat q^1, \widehat q^2)\trt \left(q^2-\widehat q^2\right) \leq 0 \qquad \forall \, (q^1,q^2) \in T,
$$
is an easy solution of the vertical model.

\bibliographystyle{spmpsci}      % mathematics and physical sciences
\bibliography{Surbib}   % name your BibTeX data base
\end{document}